\documentclass[11pt,oneside]{amsart}
\usepackage{amsmath}
\usepackage{amssymb}
\usepackage{amscd}
\usepackage{amsthm}
\usepackage{amsxtra}
\usepackage{latexsym}
\usepackage{enumitem}
\usepackage{graphicx}
\usepackage{comment}
\usepackage{tikz,tikz-3dplot}

\usepackage[all]{xy}
\numberwithin{equation}{section}
\theoremstyle{plain}
\newtheorem{theorem}[equation]{Theorem}

\newtheorem{proposition}[equation]{Proposition}

\newtheorem{conjecture}[equation]{Conjecture}

\theoremstyle{remark}

\theoremstyle{definition}

\title{Lens spaces as dual complexes of Log Calabi-Yau pairs}
\author{Morgan V Brown}
\address{Department of Mathematics, University of Miami, Coral Gables, FL 33146 USA}
\email{mvbrown@math.miami.edu}

\begin{document}

\begin{abstract}
We demonstrate the construction of singular log Calabi-Yau $4$-folds such that the dual complex of the boundary is homeomorphic to a Lens space from a log Calabi-Yau surface with action of a finite cyclic group. We explicitly obtain the Lens spaces $L(3,1)$, $L(5,1)$, and $L(5,2)$ in this way.

\end{abstract}
\maketitle
\section{Introduction}

Given a normal crossing pair $(X,\Delta)$, we may assign the dual intersection complex $\mathcal{D}(\Delta)$, whose points correspond to components of $\Delta$ and whose higher dimensional cells correspond to connected components of the intersections of divisors of $\Delta$. These complexes furnish a combinatorial invariant which encodes information relevant to many constructions in algebraic geometry giving connections to Hodge theory, birational geometry, non-archimedean geometry, and the mirror symmetry program.

The question of exactly what information about $(X,\Delta)$ can be reconstructed from $\mathcal{D}(\Delta)$ is an area of active research. In the context of birational geometry, it is known that for log crepant maps, the PL homeomorphism type of the dual complex is invariant \cite{deFernexKollarXu2012}. In Hodge theory, the $\mathbb{Q}$-cohomology groups of $\mathcal{D}(\Delta)$ can be recovered from the weight filtration on the cohomology of $X$ induced by $(X,\Delta)$ \cite{Payne}.

We will say that a pair $(X,\Delta)$ is log Calabi-Yau if $X$ is proper, $K_X+\Delta$ is numerically trival and $(X,\Delta)$ is log canonical. We associate to $(X,\Delta)$ log canonical the PL homeomorphism type of the dual complex of coefficient $1$ divisors of any dlt minimal model. See section \ref{backgd} for details.

Under mild conditions, it is expected that if $(X,\Delta)$ is log Calabi-Yau of dimension $d$, then the dual complex is homeomorphic to a quotient of a sphere $S^{d-1}$. This is known to be true for dimension $d\leq 4$ 
\cite{KollarXu}. The goal of this note is to furnish constructions of examples of log Calabi-Yau varieties of dimension $4$ such that the dual complex is a Lens space $L(3,1)$, $L(5,1)$, or $L(5,2)$.

More generally, suppose we have a log Calabi-Yau surface $(X,\Delta)$ with dual complex $S^1$, along with an action of $\mathbb{Z}/n\mathbb{Z}$ which is transitive on the dual complex generated by $\psi$. Choose $k$ relatively prime to $n$. We construct $(Y,\Gamma)$ as the quotient of $(X\times X, \pi_1^*\Delta+\pi_2^*\Delta)$ under $(\psi, \psi^k)$. The pair $(Y,\Gamma)$ is a singular log Calabi-Yau variety. We show in Thm \ref{mainhomeomorphism} that the dual complex of any dlt minimal model is homeomorphic to the Lens space $L(n,k)$. It is straightforward to find examples of log Calabi-Yau surfaces with appropriate action of $\mathbb{Z}/3\mathbb{Z}$ and $\mathbb{Z}/5\mathbb{Z}$.

However there is no representation of the group $\mathbb{Z}/p\mathbb{Z}$ as a subgroup of the automorphisms of a rational surface for a prime $p \geq 7$ \cite{Zhangauto}. This limits the ability of the construction to produce more complicated Lens spaces, but it is conceivable there is some other construction of log Calabi-Yau $4$-folds. I pose in Conjecture \ref{conjprimes} that this is not possible for sufficiently large primes $p$, with $p=7$ being the first open case.

\subsection*{Acknowledgements} The author's work was supported by Simons Foundation Collaboration Grant 524003. The author is very grateful to the Institute for Mathematics and Informatics at the Bulgarian Academy of Sciences (IMI-BAS) for hosting his visit during the completion of this project. During this visit the author was supported by the Bulgarian Ministry of Education and Science, Scientific Programme "Enhancing the Research Capacity in Mathematical Sciences (PIKOM)", No. DO1-67/05.05.2022.
%More generally, if $(X,\Delta)$ is a dlt pair, let $\Delta^{=1}$ be the union of components of $\Delta$ with coefficient $1$. Then $(X,\Delta^{=1})$ is snc in the neighborhood of the generic point of any irreducible component of the strata of $\Delta^{=1}$. Thus we may assign to the dlt pair $(X,\Delta)$ the complex $\mathcal{D}(\Delta^{=1})$. Moreover, de Fernex, Koll{\'a}r, and Xu ** show that the PL homeomorphism type of $\mathcal{D}(\Delta^{=1})$ is invariant under log crepant maps.

%Now suppose $(X,\Delta)$ is a projective log canonical pair. This allows more difficult singularities than dlt pairs, but a log canonical pair always has a minimal dlt model $(\tilde{X}, \tilde{\Delta})$, $f \colon \tilde{X} \to X$ **. By the above discussion the complex $\mathcal{D}(\Delta)^{=1}$ is up to PL homeomorphism, an invariant of $(X,\Delta)$. Our goal is to compute this invariant directly from the divisor $\Delta$.

%In general the divisor $\Delta$ may be quite singular. Unlike for dlt pairs, the irreducible components of $\Delta$ may not be normal. While they must only have nodal singularities in codimension $1$, they may have more complicated self intersections in higher codimension. 

\section{Preliminaries} \label{backgd}
We give a short background on dual complexes in the context of birational geometry following \cite{deFernexKollarXu2012, Hatcher}. A regular cell complex of dimension $m$ is defined inductively as follows. A dimension $0$ regular cell complex is a disjoint union of vertices. Given a dimension $j-1$ cell complex $C_{j-1}$, we may obtain a dimension $j$ one by attaching a collection of simplices of dimension $j$, such that the boundary of each simplex embeds into $C_{j-1}$.

For a dlt pair $(X,\Delta)$ we write $\Delta^{=1}$ to be the coefficient $1$ divisors of $\Delta$. As $(X,\Delta)$ is dlt, each component of $\Delta^{=1}$ is normal, and for any $j$ components of the $\Delta^{=1}$, the intersection is either empty or is a disjoint union of irreducible varieties of codimension $j$ in $X$. Call such a component of such an intersection a stratum $v$ of $\Delta^{=1}$. Then the dual complex $\mathcal{D}(\Delta^{=1})$ is a regular cell complex whose vertices correspond to the components of the $\Delta^{=1}$, and whose $j-1$ cells correspond to the codimension $j$ strata.

To a stratum $v$ we associate the cell $\sigma_v$. If $v$ is of codimension $j$, then there are exactly $j$ divisors containing $v$, and every stratum $v'$ properly containing $v$ can be recovered uniquely from $v$ and the subset of the $\Delta^{=1}$ containing $v'$. This furnishes the attaching map identifying the boundary of $\sigma_v$ with the union of the $\sigma_v'$.

These complexes behave well with respect to crepant birational maps: Given a map $f\colon X' \to X$, if $(X',\Delta')$ and $(X,\Delta)$ are dlt pairs, and $f^*(K_X+\Delta)\sim K_{X'} + \Delta'$ then $\mathcal{D}(\Delta^{=1})$ and $\mathcal{D}(\Delta'^{=1})$ are canonically PL-homeomorphic \cite[Prop 11]{deFernexKollarXu2012}.

If $(Y,\Gamma)$ is not dlt, we can assign to $\mathcal{D}(Y,\Gamma)$ the PL-homeomorphism type of the dual complex of the log canonical centers of any dlt minimal model. This is especially of interest when $(Y,\Gamma)$ is a log canonical pair. See \cite{BrownDC} for discussion of the problem of direct computation of the dual complex of a $3$-fold log canonical pairs.

Let $(X,\Delta)$ be dlt with $\Delta=\Delta^{=1}$, and suppose $G$ is a finite group with action on $X$ preserving the divisor $\Delta$. We have an induced action of $G$ on the dual complex $\mathcal{D}(\Delta)$, and we assume that this action is free, that is, every nontrivial element of $G$ acts with no fixed points on $\mathcal{D}(\Delta)$. To understand the quotient it is convenient to replace $\mathcal{D}(\Delta)$ with its barycentric subdivision. Such an operation corresponds to blowing up the strata of $\Delta$ in sequence by dimension, see \cite[Remark 10]{deFernexKollarXu2012}. Call the resulting pair $(X',\Delta')$. The action of $G$ extends to an action on $(X',\Delta')$, and therefore an action on the complex $\mathcal{D}(\Delta')$, which is the barycentric subdivision of $\mathcal{D}(\Delta)$. This complex $\mathcal{D}(\Delta')$ is a simplicial complex, enjoying the property that the intersection of any two cells is a face of both.

Let $(Y',\Gamma')$ be the quotient of $(X',\Delta')$ under the action of $G$.

\begin{proposition}\label{dlt}
The pair $(Y',\Gamma')$ is dlt.
\end{proposition}

\begin{proof}
Any lc center of $(Y',\Gamma')$ is the image of an lc center of $(X',\Delta')$, which means it is a stratum, call it $v$. The condition that $G$ act freely on the dual complex implies that analytically locally around the generic point of $v$, the image of $v$ is smooth and is cut out by the images of the divisors in $\Delta'$ cutting out $v$. Each of these divisors has an image in $\Delta$ of a different codimension by the barycentric subdivision construction. Hence their images are distinct in $Y'$. Thus the image of $v$ is a smooth complete intersection of the distinct components of $\Gamma'$ containing it. So $(Y',\Gamma')$ is dlt.

\end{proof}

\begin{proposition}\label{quotient}
The complex $\mathcal{D}(\Gamma')$ is identified with the quotient $\mathcal{D}(\Delta')/G$.
\end{proposition}

\begin{proof}
Let $\phi \colon X' \to Y'$ be the quotient map. We construct the map $f$ exhibiting $\mathcal{D}(\Gamma')$ as $\mathcal{D}(\Delta')/G$. For any stratum $w$ of $\Delta'$, $f$ maps $\sigma_w$ isomorphically to $\sigma_(\phi(w))$, by identifying each stratum containing $w$ with its image under $\phi$. This assignment is compatible with the face maps, and ever stratum of $\Gamma'$ is the orbit of a stratum of $\Delta'$, so $f$ gives a surjective map from $\mathcal{D}(\Delta') \to \mathcal{D}(\Gamma')$. Finally, note that two points of $\mathcal{D}(\Delta')$ have the same image under $f$ if and only if they are related by an element of $G$, hence $\mathcal{D}(\Gamma')$ is identified with $\mathcal{D}(\Delta')/G$. 
\end{proof}

We give a brief review of the construction of Lens spaces \cite{Bredon}: Consider the sphere $S^3$ as the locus in $\mathbb{C}^2$ cut out by $|z|^2+|w|^2=1$. Consider positive integers $k,n$ with $k$ relatively prime to $n$. Then the action of $\mathbb{Z}/n\mathbb{Z}$ on $S^3$ induced by the generator $(z,w)\to (e^{2\pi \mathrm{i} /n }z, e^{2\pi k \mathrm{i}/n}w)$ is fixed point free. The topological quotient is an orientable $3$-manifold called a Lens space, denoted $L(n,k)$. The Lens spaces $L(5,1)$ and $L(5,2)$ are a well known example of two manifolds of the same dimension with isomorphic fundamental and homology groups but not homotopy equivalent \cite[VI, Thm 10.15]{Bredon}.

%The Lens spaces can also be constructed through topological surgery and as a branched double cover of the sphere $S^3$.

\section{Main Construction} \label{mainc}

Define $M_n$ as a simplicial complex on points $\Delta_{1,1}, \ldots \Delta_{1,n}, \Delta_{2,1} \ldots \Delta_{2,n}$. A simplex $\sigma$ belongs to $M_n$ if and only if for all $i,j\in \mathbb{Z}/n\mathbb{Z}$ whenever $\Delta_{1,i}$ and $\Delta_{1,j}$ are in $\sigma$ or $\Delta_{2,i}$, and $\Delta_{2,j}$ are in $\sigma$, then $i,j$ are consecutive in the group $\mathbb{Z}/n\mathbb{Z}$, that is either $i=j+1$ or $j=i+1$.

Let $(X,\Delta)$ be a surface pair such that the components of $\Delta$ form an $n$ cycle. Then the product $(X\times X,\pi_1^*(\Delta)+\pi_2^*(\Delta))$ has dual complex identified with $M_n$.

Now we identify the topological realization $|M_n|$ with a subset $L_n$ of Euclidean space $\mathbb{R}^{2n}$, with coordinates $Y_{1,1}, \ldots Y_{1,n}, Y_{2,1}, \ldots Y_{2,n}$. Let $S_n$ be the simplex defined by $Y_{\mu,i}\geq 0$ and $\sum Y_{\mu,i}=1$. Each point $\Delta^{\mu, i}$ of the simplicial complex has a corresponding coordinate $Y_{\mu,i}$. For a simplex $\sigma$, we identify $|\sigma|$ with the locus of points in $S_n$ satisfying $Y_{\mu,i}=0$ if $\Delta_{\mu,i}\notin \sigma$. This identification is compatible with the face relations so we have identified $|M_n|$ with a subset of $S_n$, which we call $L_n$.

Our next goal is to define a homeomorphism between $L_n$ and the three sphere in $\mathbb{C}^2$ defined by $|z|^2+|w|^2=1$. For $x_1, x_2$ satisfying $x_1+x_2\leq 1, x_1 \geq 0, x_2, \geq 0$, we define 
\[
f_n(x_1,x_2)=\sin(\pi/2(x_1+x_2))e^{(2\pi \mathrm{i}/n)(x_2/(x_1+x_2))})
\], if at least one of the $x_i \neq 0$, and define $f(0,0)=0$. We see that $f_n$ is a continuous function with complex codomain. Now, for a point $P$ of $L_n$, there exist $i,j$ such that $Y_{1,i}+Y_{1,i+1}+Y_{2,j}+Y_{2,j+1}=1$ and so all other coordinates must be zero. Define 

\[
F^{i,j}_n(P)=(e^{2\pi\mathrm{i}\cdot(i/n)}f_n(Y_{1,i},Y_{1,i+1}), e^{2\pi\mathrm{i}\cdot(j/n)}f_n(Y_{2,j},Y_{2,j+1}))
\]
\begin{proposition}
The $F^{i,j}_n$ extend to a homeomorphism $F_n \colon L_n \to S^3$.
\end{proposition}

\begin{proof}
First, we note that for any point $P$ of $L_n$, the function $F_n$ is well defined. The only possible ambiguity in the first coordinate $z$ of $F_n$ is if exactly one of the $Y_{1,i}$ are nonzero. But in this case whichever option we use we obtain $z=\sin(\pi/2(Y_{1,i}))e^{2\pi\mathrm{i}\cdot(i/n)}$.

Next, we note that the output $(z,w)$ of $F_n$ always satisfies $|z_1|^2+|z_2|^2=1$. So we may restrict the codomain of $F_n$ to $S^3$.

To show that $F_n$ is a homeomorphism, we make use of the well known fact that a continuous bijection from a compact space to a Hausdorff space is a homeomorphism. As $f_n$ was continuous, so is $F_n$, so we need only show that $F_n$ is a bijection. We do so by describing how to construct the unique preimage of any point on the unit sphere $|z_1|^2+|z_2|^2=1$.

If $z_\mu \neq 0$, then $\frac{z_\mu}{|z_\mu|}$ is a point on the unit circle. If it is the root of unity $e^{(2\pi \mathrm {i})(j/n)}$, then $Y_{\mu,j}$ is nonzero but all the other $Y_{\mu,j'}$ are $0$. If the phase of $\frac{z_\mu}{|z_\mu|}$ lies strictly between the roots of unity numbered $j$ and $j+1$, then $Y_{\mu,j}$ and $Y_{\mu,j+1}$ are nonzero and their ratio can be uniquely recovered from the phase.

If the other coordinate $z_\lambda=0$, then all $Y_{\lambda,j}=0$ and we have uniquely recovered a point of $L_n$.
Both $z_1$ and $z_2$ cannot be $0$.

If $z_1$ and $z_2$ are nonzero, we have seen that we can recover uniquely from the phases the ratios of $Y_{1,i}$ to $Y_{1,i+1}$ and that of $Y_{2,j}, Y_{2,j+1}$. The quantity $Y_{1,i}+Y_{1,i+1}$ is recovered uniquely from the magnitude of $z_1$, and this is enough information to determine the rest.
\end{proof}

Now, let Let $(X,\Delta)$ be an snc log Calabi-Yau surface, such that the dual complex of $\Delta$ is an $n$-cycle. Let $\psi$ be an order $n$ automorphism of $X$ which cyclically permutes the components of $\Delta$. We assume also that $\psi$ fixes no divisor of $\Delta$. Let $X\times X$ be the product, and let $\Delta_1, \Delta_2$ be the two pullbacks of $\Delta$. Choose $1\leq k <n$ relatively prime to $n$, and let $\Psi_k$ be the automorphism given by $\psi$ on the first factor and $\psi^k$ on the second.

Define $(Y,\Gamma)$ to be the log canonical pair given by the quotient of $(X\times X, \Delta_1+\Delta_2)$ by $\mathbb{Z}/n\mathbb{Z}$ generated by $\Psi_k$. As $\Psi_k$ fixes no divisor of $X\times X$, $K_Y+\Gamma$ is numerically trivial.

\begin{theorem}\label{mainhomeomorphism}
The dual complex of a dlt minimal model of $(Y,\Gamma)$ is homeomorphic to the Lens space $L(n,k)$.
\end{theorem}

\begin{proof}
We identify the dual complex of $(X\times X, \Delta_1+\Delta_2)$ with $L_n$, and the action of $\Psi_k$ sends $Y_{1,i}$ to $Y_{1,i+1}$, and $Y_{2,j}$ to $Y_{2,j+k}$. We apply $F_n$ to identify this action with an action on $S^3$. Let $(z_1,z_2)=F_n(P)$. Then
\[F_n(\Psi_k(P))=(e^{2\pi \mathrm{i} /n }z_1, e^{2\pi k \mathrm{i}/n}z_2).
\]
The quotient of $S^3$ under this action is the Lens space $L(n,k)$.

Let $(Z, \Delta')$ be the iterated blow up of $(X\times X, \Delta_1+\Delta_2)$, starting with the $0$-dimensional strata and at the $i$th stage blowing up along the strict transform of the dimension $i$ strata of $\Delta_1+\Delta_2$. We identify $\mathcal{D}(\Delta')$ with the barycentric subdivision of $\mathcal{D}(\Delta_1+\Delta_2)$. Then by Proposition \ref{dlt}, the quotient $(Y',\Gamma')$ of $(Z,\Delta')$ is dlt, and by Proposition \ref{quotient}, $\mathcal{D}(\Gamma')$ is identified with $L(n,k)$. The pair $(Y,\Gamma')$ is a dlt minimal model of $(Y,\Gamma)$.
\end{proof}

We can construct the Lens space $L(3,1)$ by taking $X$ to be $\mathbb{P}^2$, $D$ the sum of the three coordinate lines and letting $\mathbb{Z}/3\mathbb{Z}$ act by cyclic action on the coordinates.

For $L(5,1)$, and $L(5,2)$, we consider $X=\mathbb{P}^2$ blown up at $4$ points. Take $L$ for the class of a pullback of a line, and $\Delta_i$ for the exceptional divisors.

Then the following $5$ classes each admit a $-1$ curve, and together form an anticanonical cycle, such that the dual complex is the standard cycle:
\begin{enumerate}
\item $\Delta_1=L-E_2-E_4$
\item $\Delta_2=L-E_1-E_3$
\item $\Delta_3=E_3$
\item $\Delta_4=L-E_2-E_3$
\item $\Delta_5=E_2$
\end{enumerate}

In fact there is an action of $\mathbb{Z}/5\mathbb{Z}$ on the pair $(X,\Delta)$ which is transitive on the $\Delta_i$. This comes from the realization of $X$ as the moduli space $\overline{\mathcal{M}_{0,5}}$ \cite{Kapranov}. This moduli space admits the action of $\mathbb{Z}/5\mathbb{Z}$ via the cyclic permutation of the $5$ marked points on $\mathbb{P}^1$.

Each $-1$ curve in $X$ corresponds to a locus $A_{i,j}$ corresponding to stable rational curves with the marked points $p_i$ and $p_j$ on the same component with exactly one other special point, corresponding to where that component intersects the next component of the stable curve. The induced action on the $A_{i,j}$ apply the cyclic permutations to the indices $1, \ldots 5$. To show the action is transitive on $\mathcal{D}(\Delta)$, we must identify the $\Delta_i$ with the $A_{i,j}$.

\begin{proposition}
The locus $\Delta_i$ corresponds to the locus $A_{i, i+2}$
\end{proposition}
\begin{proof}
We interpret the map to $\mathbb{P}^1$ induced by the pencil of conics through the $4$ points as the total family $\overline{\mathcal{M}_{0,5}}\to \overline{\mathcal{M}_{0,4}}$ \cite{Kapranov}. The four sections represent the four marked points, and the point of $X$ itself is the location of the fifth point. When this point is on one of the $E_i$, we must create a new component of the fiber containing both $p_i$ and $p_5$. Thus $E_i=A_{i,5}$. On the other hand, for the other components each $L-E_i-E_j$ corresponds to the $A_{k,l}$ where $k$ and $l$ are the remaining indices from the first four in order.

We can then confirm our list of $\Delta_i$ correspond exactly the the $A_{i,i+2}$.
\end{proof}

As a result applying Theorem \ref{mainhomeomorphism} we may produce a cyclic action on $(X\times X, \pi_1^*(\Delta)+\pi_2^*(\Delta))$ such that the resulting log canonical log Calabi-Yau pair has dual complex which is a Lens space, either $L(5,1)$ or $L(5,2)$ depending on how we construct the action.

\section{Comments}\label{Discussion}

Our examples of log Calabi-Yau $4$-fold pairs with dual complex $L(p,q)$ require a surface log Calabi-Yau pair with action of $\mathbb{Z}/p\mathbb{Z}$. The condition of having a rational surface with an action of a finite cyclic group is quite restrictive. If we restrict to groups of prime order such actions exist only for $p\leq 5$ \cite{Zhangauto}.

In higher dimensions not every rationally connected variety is rational, so we cannot even use the study of the Cremona group to rule out examples. Still, I expect not every Lens space can be constructed as the dual complex of a log Calabi-Yau $4$-fold pair.

\begin{conjecture}\label{conjprimes}
For sufficiently large prime $p$, there is no dlt log-Calabi-Yau $4$-fold pair whose dual complex is homeomorphic to $L(p,k)$ for any $k$.
\end{conjecture}
The first open case is $p=7$.

Our examples always admit a dlt minimal model $(Y',\Gamma')$ where every component of $\Gamma'$ has coefficient $1$. As Lens spaces are orientable, by \cite[Prop 5.5]{filipazzi2022indexcoregularityzerolog} we must have $K_{Y'}+\Gamma'\sim 0$. On the other hand the quotient construction implies $Y'$ admits a torsion divisor whose algebra of sections induces the cover $X'$, likewise for $Y$.

Any example of a dlt $4$-fold pair $(Y',\Gamma')$ with $\mathcal{D}(\Gamma')$ homeomorphic to $L(p,q)$ must admit a quasi-{\'e}tale cover $(\tilde{Y'},\tilde{\Gamma'})$, corresponding to the universal cover of $L(p,q)$ by the sphere $S^3$, by \cite[Thm 2.5]{KollarXu}.

\bibliographystyle{amsalpha}      
\bibliography{References}

\providecommand{\bysame}{\leavevmode\hbox to3em{\hrulefill}\thinspace}
\providecommand{\MR}{\relax\ifhmode\unskip\space\fi MR }
\providecommand{\MRhref}[2]{%
  \href{http://www.ams.org/mathscinet-getitem?mr=#1}{#2}
}
\providecommand{\href}[2]{#2}
\begin{thebibliography}{dFKX17}

\bibitem[Bre93]{Bredon}
Glen~E. Bredon, \emph{Topology and geometry}, Graduate Texts in Mathematics,
  vol. 139, Springer-Verlag, New York, 1993. \MR{1224675}

\bibitem[Bro22]{BrownDC}
Morgan~V. Brown, \emph{The dual complex of a semi-log canonical surface}, Int.
  Math. Res. Not. IMRN (2022), no.~10, 7495--7515. \MR{4418713}

\bibitem[dFKX17]{deFernexKollarXu2012}
T.~{de}~{Fernex}, J.~Koll{\'a}r, and C.~Xu, \emph{The dual complex of
  singularities}, Higher dimensional algebraic geometry, in honour of Professor
  Yujiro Kawamatas 60th birthday, vol.~74, Adv. Stud. Pure Math., December
  2017, pp.~103--130.

\bibitem[FMM22]{filipazzi2022indexcoregularityzerolog}
Stefano Filipazzi, Mirko Mauri, and Joaquín Moraga, \emph{Index of
  coregularity zero log calabi-yau pairs}, 2022.

\bibitem[Hat02]{Hatcher}
A.~Hatcher, \emph{Algebraic topology}, Cambridge University Press, Cambridge,
  2002.

\bibitem[Kap93]{Kapranov}
M.~M. Kapranov, \emph{Veronese curves and {G}rothendieck-{K}nudsen moduli space
  {$\overline M_{0,n}$}}, J. Algebraic Geom. \textbf{2} (1993), no.~2,
  239--262. \MR{1203685}

\bibitem[KX16]{KollarXu}
J.~Koll\'ar and C.~Xu, \emph{The dual complex of {C}alabi-{Y}au pairs}, Invent.
  Math. \textbf{205} (2016), no.~3, 527--557.

\bibitem[Pay13]{Payne}
Sam Payne, \emph{Boundary complexes and weight filtrations}, Michigan Math. J.
  \textbf{62} (2013), no.~2, 293--322. \MR{3079265}

\bibitem[Zha02]{Zhangauto}
D.-Q. Zhang, \emph{Automorphisms of finite order on {G}orenstein del {P}ezzo
  surfaces}, Trans. Amer. Math. Soc. \textbf{354} (2002), no.~12, 4831--4845.
  \MR{1926853}

\end{thebibliography}
\end{document}